\documentclass{amsart}

\usepackage[style=trad-plain,backref=true,url=true,isbn=false,alldates=year,sortcites=true,maxnames=9]{biblatex}
\renewbibmacro{in:}{%
  \ifentrytype{article}{}{\printtext{\bibstring{in}\intitlepunct}}}

\usepackage{mathptmx}
\usepackage{amssymb,amsmath,amsthm,amsfonts}
\usepackage{graphicx}
\usepackage{bm}
\usepackage{xcolor}
\usepackage{xurl}
\usepackage{mathrsfs}
\usepackage[shortlabels]{enumitem}
\usepackage{caption}

\usepackage{MnSymbol}
\definecolor{linkblue}{HTML}{003d73}
\definecolor{linkgreen}{HTML}{006161}
\definecolor{linkred}{HTML}{a11950}
\usepackage{hyperref}
\hypersetup{
	pdftitle={Trefoil Probabilities and Polyhedra},
	pdfauthor={Joe Geisz, Chris Peterson, and Clayton Shonkwiler},
	pdfsubject={knot theory},
	pdfkeywords={stick knots, polyhedra, combinatorics, probability},
	colorlinks=true,
	linkcolor=linkblue,
	citecolor=linkgreen,
	urlcolor=linkred
}
\usepackage{pifont}
\usepackage[margin=1in]{geometry}
\usepackage{mathtools}
\usepackage{booktabs}
\usepackage{mathdots}

\usepackage{enumitem}

\usepackage[nameinlink,capitalise]{cleveref}

\newtheorem{thm}{Theorem}[section]
\newtheorem{lem}[thm]{Lemma}
\newtheorem{prop}[thm]{Proposition}

\theoremstyle{definition}

\newtheorem{remark}[thm]{Remark}

\AddToHook{env/lem/begin}{\crefalias{thm}{lem}}
\AddToHook{env/prop/begin}{\crefalias{thm}{prop}}

\crefname{lem}{Lemma}{Lemmas}
\crefname{prop}{Proposition}{Propositions}

\makeatletter
\renewcommand*\env@matrix[1][*\c@MaxMatrixCols c]{%
  \hskip -\arraycolsep
  \let\@ifnextchar\new@ifnextchar
  \array{#1}}
\makeatother

\newcommand{\R}{\mathbb{R}}

\newcommand{\potential}[1]{%
    {\def\tmp{#1}
    \ifx\tmp\empty
        {E_{\vec{r}}}
    \else
        {E_{#1}}
    \fi}}

\newcommand{\seqnum}[1]{\href{http://oeis.org/#1}{#1}}

\newcommand{\clay}[1]{\textcolor{red}{Clay says: {#1}}}

\makeatletter
\newcommand{\subalign}[1]{%
  \vcenter{%
    \Let@ \restore@math@cr \default@tag
    \baselineskip\fontdimen10 \scriptfont\tw@
    \advance\baselineskip\fontdimen12 \scriptfont\tw@
    \lineskip\thr@@\fontdimen8 \scriptfont\thr@@
    \lineskiplimit\lineskip
    \ialign{\hfil$\m@th\scriptstyle##$&$\m@th\scriptstyle{}##$\hfil\Krcr
      #1\Krcr
    }%
  }%
}
\makeatother

\DefineBibliographyStrings{english}{%
  backrefpage = {$\uparrow$},
  backrefpages = {$\uparrow$},
  page = {p\adddot},
  pages = {pp\adddot},
}

\DeclareSourcemap{
  \maps{
    \map{
      \step[fieldset=pagetotal, null]
	  \step[fieldset=pubstate, null]
    }
  }
}

\renewbibmacro*{doi+eprint+url}{%
 \iftoggle{bbx:doi}
   {\printfield{doi}}
   {}%
 \newunit\newblock
 \iftoggle{bbx:eprint}
   {%
     \iffieldundef{doi}%
       {\usebibmacro{eprint}}%
       {}}
 \newunit\newblock
 \iftoggle{bbx:url}
   {%
     \iffieldundef{doi}%
	   {%
	     \iffieldundef{eprint}%
		   {\printfield{url}}}}
   {}}

\DeclareFieldFormat{eprint:urn}{%
  \mkbibacro{URN}\addcolon\space
  \ifhyperref
    {\href{https://nbn-resolving.org/urn:#1}{\nolinkurl{#1}}}
    {\nolinkurl{#1}}}
\DeclareFieldAlias{eprint:URN}{eprint:urn}
\DeclareFieldFormat{eprint:hal}{%
  \mkbibacro{HAL}\addcolon\space
  \ifhyperref
    {\href{https://hal.science/#1}{\nolinkurl{#1}}}
    {\nolinkurl{#1}}}
\DeclareFieldAlias{eprint:HAL}{eprint:hal}
\DeclareFieldFormat{eprint:numdam}{%
  Numdam\addcolon\space
  \ifhyperref
    {\href{http://www.numdam.org/item/#1}{\nolinkurl{#1}}}
    {\nolinkurl{#1}}}
\DeclareFieldAlias{eprint:Numdam}{eprint:numdam}
\DeclareFieldFormat{eprint:ark}{%
  \mkbibacro{ARK}\addcolon\space
  \ifhyperref
    {\href{https://n2t.net/ark:#1}{\nolinkurl{#1}}}
    {\nolinkurl{#1}}}
\DeclareFieldAlias{eprint:ARK}{eprint:ark}
\DeclareFieldFormat{eprint:zbl}{%
  Zbl\addcolon\space
  \ifhyperref
    {\href{https://zbmath.org/#1}{\nolinkurl{#1}}}
    {\nolinkurl{#1}}}
\DeclareFieldAlias{eprint:Zbl}{eprint:zbl}
\DeclareFieldFormat{eprint:mr}{%
  \mkbibacro{MR}\addcolon\space
  \ifhyperref
    {\href{https://mathscinet.ams.org/mathscinet-getitem?mr=#1}{\nolinkurl{#1}}}
    {\nolinkurl{#1}}}
\DeclareFieldAlias{eprint:MR}{eprint:mr}

\newsavebox\myrhs
\newlength\myrhswd
\newlength\myrhssepwd
\sbox\myrhs{%
  \parbox[t]{\myrhswd}{%
    \vspace{-.2cm}
    \begin{flushright} 
      {\bf \textsf{\footnotesize SAND2026-26663O}} %
    \end{flushright}%
  }%
}
\AddToHookNext{shipout/background}{
  \put(\dimexpr\paperwidth-\myrhswd-\myrhssepwd\relax,-\baselineskip){%
    \usebox\myrhs
  }%
}

\begin{document}

\title{Trefoil Probabilities and Polyhedra}

\author{Joseph Geisz}
\address{Department of Mathematics, Colorado State University, Fort Collins, Colorado 80523}
\email{joe.geisz@colostate.edu}
\address{Sandia National Laboratories, 1515 Eubank SE
Albuquerque, New Mexico 87123}
\email{jkgeisz@sandia.gov}

\author{Chris Peterson}
\address{Department of Mathematics, Colorado State University, Fort Collins, Colorado 80523}
\email{christopher2.peterson@colostate.edu}
\thanks{The second author was partially supported by NSF grant ATD–2428052.}

\author{Clayton Shonkwiler}
\address{Department of Mathematics, Colorado State University, Fort Collins, Colorado 80523}
\email{clayton.shonkwiler@colostate.edu}



\begin{abstract} We determine the exact probability that the closed hexagonal polygon obtained by cyclically joining six independent points uniformly distributed on the unit sphere is knotted. The only possible nontrivial knot is a trefoil. Almost surely, the convex hull of the six points has one of two simplicial combinatorial types: a combinatorially regular octahedron and a combinatorially non-regular octahedron. We show that the straight-line complete graph on the six vertices of the hull of the combinatorially regular octahedral type admits exactly one unoriented trefoil Hamiltonian cycle. On the other hand, no cycle connecting vertices of the irregular octahedron can produce a trefoil. We then use stereographic projection to transform the probability of the regular hull type to a Sylvester-type problem for the planar beta-prime probability measure $d\mu(x,y)=\frac{dx \, dy}{\pi\ (1+x^2+y^2)^2}$. Applying Stokes' theorem and the Blaschke–Petkantschin formula, we compute the expected squared $\mu$-content of a random triangle. This yields a regular octahedral probability of $\frac{15}{4\pi^2}$ and a trefoil probability of $\frac{1}{16\pi^2}$.
\end{abstract}

\maketitle

\section{Introduction}
Let $x_1, \dots, x_6$ be six points chosen uniformly at random from the unit sphere $\mathbb{S}^2$. Connecting consecutive points (in cyclic fashion, so that $x_1$ follows $x_6$) by straight segments produces a polygonal closed curve which is embedded with probability 1. The goal of this paper is to answer the following question: what is the probability that this closed curve is knotted?

Generating random knots by connecting random points on the sphere is an example of a \emph{random jump model}~\cite[\S 3.4]{even-zoharModelsRandomKnots2017}, in which vertices of a polygonal knot are generated by independent draws from some probability distribution on~$\R^3$. These models of random knots go back at least to Millett's work in the late 1990s~\cite{millettMonteCarloExplorations2000}, in which they were proposed to incorporate confinement constraints on knotted biomolecules; see also~\cite{arsuagaKnottingProbabilityDNA2002,arsuagaDNAKnotsReveal2005,arsuagaSamplingLargeRandom2007}. While numerical experiments suggest that unknots become exponentially unlikely as the number of vertices goes to infinity~\cite{millettMonteCarloExplorations2000,arsuagaSamplingLargeRandom2007}, we are not aware of any exact results about knot probabilities for any finite $n$ in such models (though some bounds for related models are known; e.g.,~\cite{ramirezalfonsinSpatialGraphsOriented1999a,hakeKnottingProbabilityEquilateral2019}).

For hexagonal knots, only two knot types are possible: the unknot and the trefoil knot~\cite{negamiRamseyTheoremsKnots1991,randellInvariantsPiecewiselinearKnots1998,millettKnottingRegularPolygons1994,knotinfo}.\footnote{Strictly speaking, there are two distinct trefoils: the right-handed trefoil and the left-handed trefoil. In our model, the two are equiprobable.} So an equivalent question to the one asked above is: given $x_1, \dots, x_6$ chosen uniformly from  $\mathbb{S}^2$, what is the probability $p$ that the corresponding polygonal knot forms a trefoil knot?

\begin{thm}
    If six ordered points are chosen independently and uniformly at random from the unit sphere and cyclically connected by straight line segments to form a closed six-sided space polygon in $\mathbb R^3$, then the probability, $p$, that this polygon is a trefoil knot is
    \[
        p = \frac{1}{16\pi^2} \approx 0.00633257.
    \]
\label{thm:probtrefoil}
\end{thm}

The strategy for proving this theorem is to focus on the convex hull $\mathcal{H}$ of $x_1, \dots , x_6$. With probability 1,  $\mathcal{H}$ will be an octahedron, but there are two different combinatorial types of octahedra that arise: a (combinatorially) regular octahedron, with vertex degrees $(4,4,4,4,4,4)$, and a (combinatorially) irregular octahedron with vertex degrees $(3,3,4,4,5,5)$; see~\cref{fig:topotypes}. Throughout, `regular' and `irregular' refer to these two combinatorial types.

\begin{figure}[h!]
    \centering
    \includegraphics[width=0.3\linewidth]{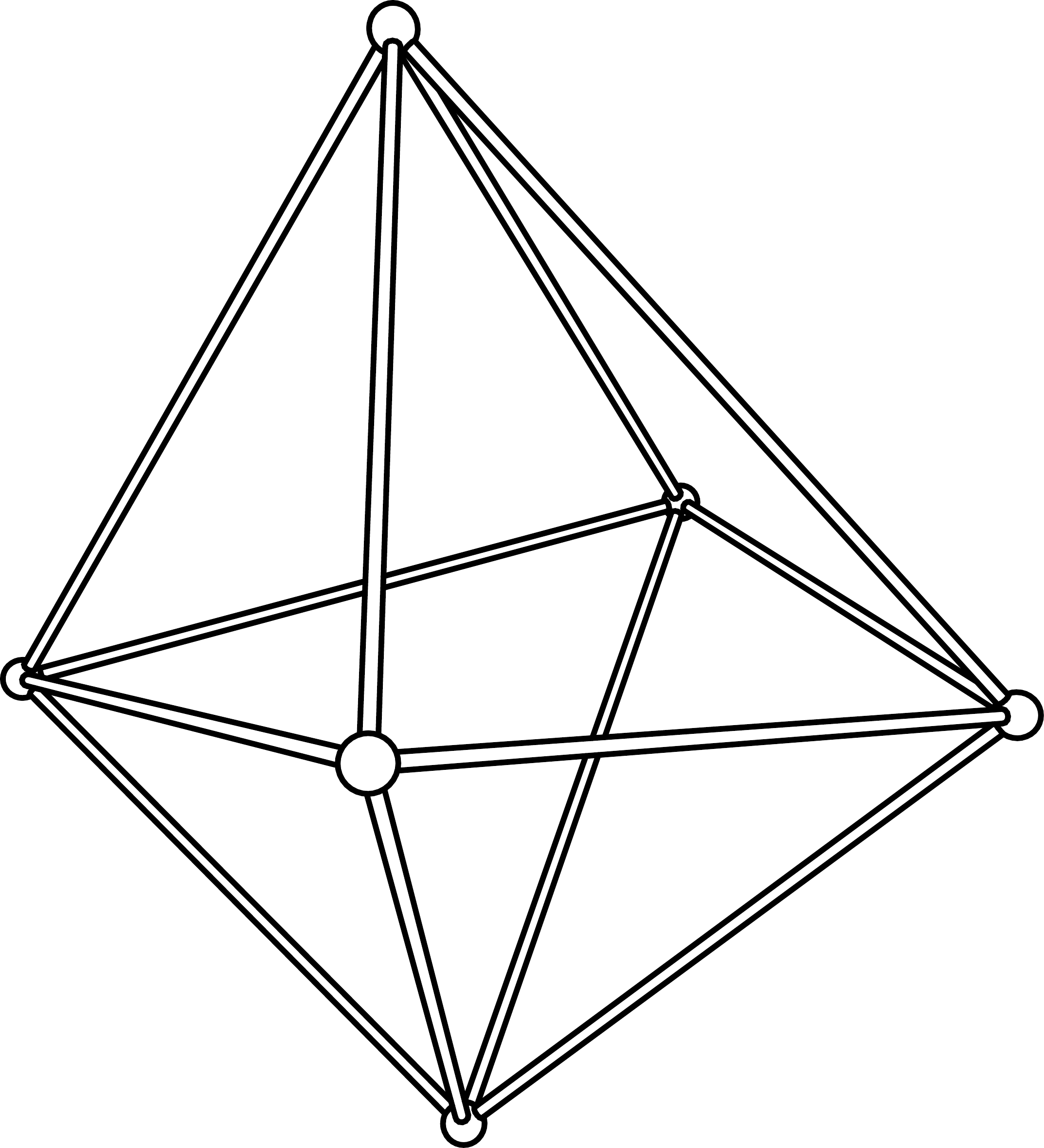}\qquad
    \includegraphics[width=0.26\linewidth]{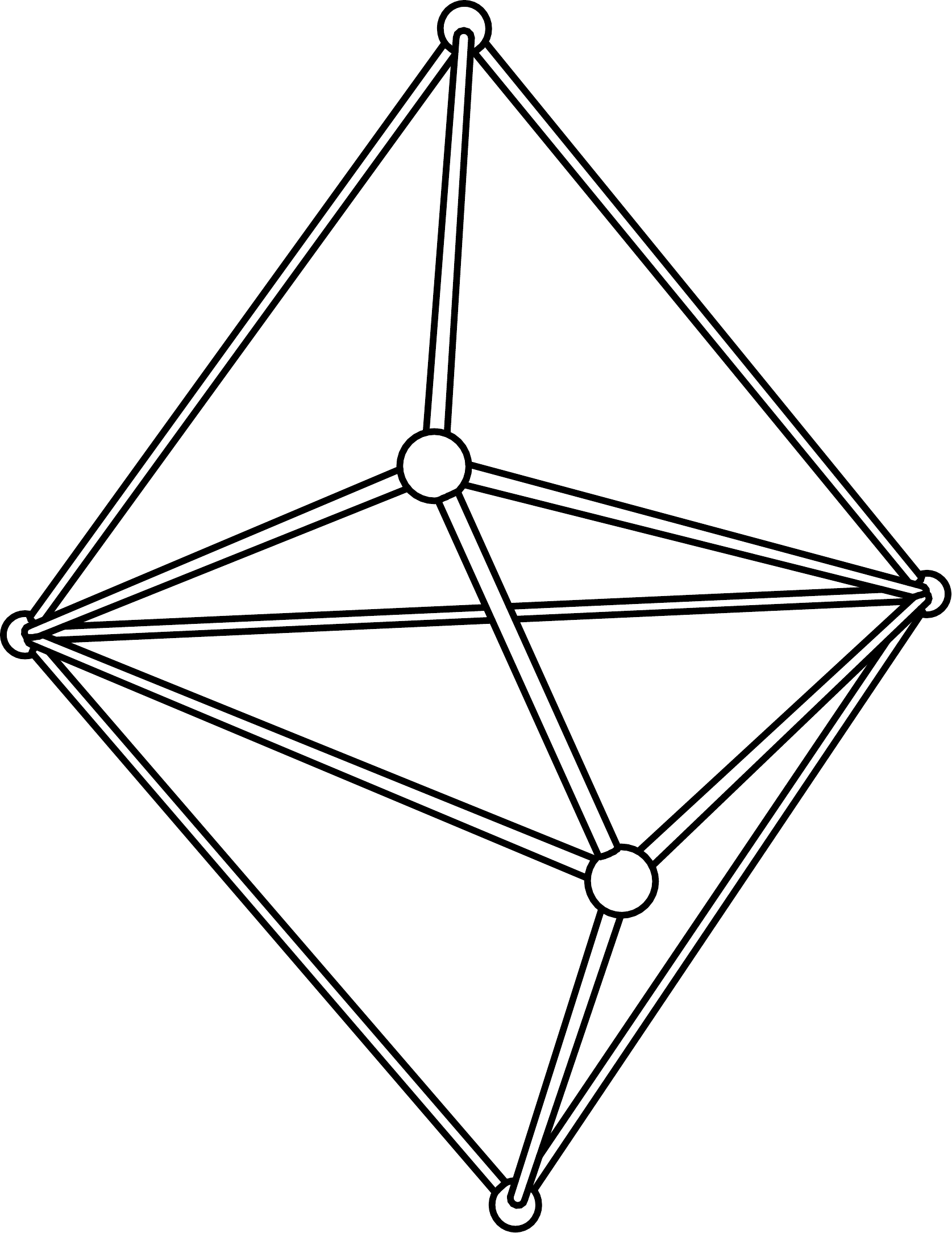}
    \caption{Examples of the two distinct types of octahedra. On the left, a (combinatorially) regular octahedron, and on the right, a (combinatorially) irregular octahedron. 
    }
    \label{fig:topotypes}
\end{figure}

Permuting $x_1, \dots , x_6$ does not change $\mathcal{H}$, and each permutation gives a labeled oriented tour of the vertices of $\mathcal{H}$. As we will see in \cref{sec:knots and polyhedra}, when $\mathcal{H}$ is a regular octahedron exactly 12 of the 720 permutations yield a trefoil knot, and the rest yield an unknot. On the other hand, when $\mathcal{H}$ is the irregular octahedron all permutations produce an unknot. Therefore, $p = \frac{12}{720}q = \frac{q}{60}$, where $q$ is the probability that the convex hull is a regular octahedron, and \cref{thm:probtrefoil} is an immediate corollary of the following theorem.

\begin{thm}
    If six points are chosen independently and uniformly at random from the unit sphere, the probability that their convex hull is a combinatorially regular octahedron is
    \[
        q = \frac{15}{4\pi^2} \approx 0.379954.
    \]
    \label{thm:q}
\end{thm}

This answers a question asked on Mathematics Stack Exchange by Jeppe Stig Nielsen~\cite{4880954}. This is a Sylvester-type problem, so called in honor of an 1864 question posed by Sylvester~\cite{sylvesterMathematicalQuestion14911864}. In modern language, Sylvester asked the probability that 4 random points in the plane were all on the boundary of their convex hull. This and related questions sparked heated debate~\cite{demorganSolutionProblemDetermining1866,sylvesterMathematicalQuestion18051865,sylvesterMathematicalQuestion18091865,woolhouseSolutionMathematicalQuestion1868,whitworthRandomNo1868,wilsonFourpointProblemSimilar1866,woolhouseObservationsThreeFour1866,sylvesterMathematicalQuestion18461865,woolhouseAdditionalObservationsFourpoint1867,sylvesterEditorsNote1865,sylvesterAlgebraicalResearchesContaining1864,demorganInfinitySignEquality1871,inglebyProblemTheoryChances1866,inglebyProblemTheoryChances1866,whitworthNoteDrInglebys1866}, highlighting the inadequacy of contemporary understanding of randomness and probability theory; see the survey by Pfiefer~\cite{pfieferHistoricalDevelopmentSylvesters1989} for an introduction to the original Sylvester problem, Cantarella et al.'s paper~\cite{cantarellaRandomTrianglesPolygons2019} for a modern solution to Sylvester's problem, and the paper of Gusakova and Kabluchko~\cite{gusakovaSylvestersProblemBetatype2025a} (and the many references therein) for some generalizations to higher dimensions.

Inspired by Felix Pahl~\cite{jorikiAnswer2024}, our strategy for proving \cref{thm:q} is to stereographically project from $x_6$ to the plane $x_6^\perp$. The images of the points $x_1, \dots , x_5$ have density $\frac{1}{\pi(1+x^2+y^2)^2}$, and their convex hull $\mathcal{C}$ will be a triangle if and only if $x_6$ was a vertex of $\mathcal{H}$ of degree three, which of course implies that $\mathcal{H}$ must have been an irregular octahedron. 

Letting $p_k$ be the probability that $\mathcal{C}$ has $k$ vertices, it will follow (see \cref{lem:q and p3}) that $p_3 = \frac{1-q}{3}$. Consequently, \cref{thm:q} is itself a corollary of the following.

\begin{thm}
\label{thm:beta vertex count probs}
If 5 points in the plane are chosen independently at random with respect to the measure $d\mu := \frac{1}{\pi(1+x^2+y^2)^2}d\lambda$, where $\lambda$ is Lebesgue measure on the plane, then the probabilities $p_k$ are
\[
    p_3 = p_5 = \frac{1}{3} - \frac{5}{4\pi^2} \approx 0.20668 \quad \text{and} \quad p_4 = \frac{1}{3} + \frac{5}{2\pi^2} \approx 0.58664.
\]
\end{thm}

\Cref{thm:beta vertex count probs} may be of some independent interest, as the measure $\mu$ is an example of a beta-prime measure introduced by Miles and Ruben~\cite{milesIsotropicRandomSimplices1971,rubenCanonicalDecompositionProbability1980} and $\mathcal{C}$ is an example of a beta-prime polygon~\cite{affentrangerConvexHullRandom1991,buchtaStochasticalApproximationConvex1985,gusakovaSylvestersProblemBetatype2025a,kabluchkoExpectedIntrinsicVolumes2019,kabluchkoRandomSimplicesBetaType2026}. This also seems to be the first example where such probabilities are computed for more than 4 points chosen from a probability distribution on the plane which is not compactly supported. Using~\cite{282148}, these probabilities are known for points chosen uniformly from a triangle~\cite{valtrProbabilityThatnRandom1996}, a parallelogram~\cite{valtrProbabilityThatnRandom1995}, and a disk~\cite{marckertProbabilityThat$n$2017,milesIsotropicRandomSimplices1971}, but in general these three probabilities seem to be known for only a few examples. The fact that $p_3 = p_5$ is an interesting feature of the model of random convex polygons that we consider; by contrast, the corresponding probabilities for points chosen from a triangle, a parallelogram, or a disk are $p_3^\triangle = \frac{5}{36}$, $p_4^\triangle = \frac{5}{9}$, $p_5^\triangle = \frac{11}{36}$, and $p_3^\square = \frac{5}{48}$, $p_4^\square = \frac{5}{9}$, $p_5^\square = \frac{49}{144}$, and $p_3^{\bigcircle}=\frac{15}{16\pi^2}, p_4^{\bigcircle}=\frac{65}{12\pi^2}, p_5^{\bigcircle}=1-\frac{305}{48\pi^2}$.

Since $p_3$ is ten times the probability that the triangle determined by the first three points contains the remaining two, our  strategy for proving \cref{thm:beta vertex count probs}—carried out in \cref{sec:calculation}—is to compute the expected squared probability content of the triangle determined by three random points drawn from $\mu$. Using Stokes' theorem, this will reduce to a two-point integral which we can compute exactly using the Blaschke–Petkantschin formula.




\section{Knots and polyhedra}
\label{sec:knots and polyhedra}

The goal in this section is to prove the assertion that computing the trefoil probability is equivalent to computing the probability that the convex hull $\mathcal{H}$ of six random points on the sphere is a regular octahedron.

First, we need to know what the possible combinatorial types of $\mathcal{H}$ are. Bowen and Fisk algorithmically enumerated all triangulations of the sphere with $n\leq12$ vertices \cite{bowenGenerationsTriangulationsSphere1967}. Particularly relevant for our case is the following fact:

\begin{prop}[{Bowen–Fisk~\cite{bowenGenerationsTriangulationsSphere1967}; see also~\cite{michonEnumeration2000}}]
For any 6 points on $\mathbb{S}^2$ which are in general position (i.e. no 4 points are coplanar), the convex hull of the points is one of two types of polyhedra. The types can be distinguished by their vertex degrees. The regular octahedron has vertex degrees (4,4,4,4,4,4) and the irregular octahedron  has vertex degrees (3,3,4,4,5,5). 
\label{prop:types}
\end{prop}


Since uniform points on the sphere will be generic with probability 1, these are the only two types that appear with positive probability in our problem. For either type of convex polyhedron, we will show that the number of unoriented trefoil Hamiltonian cycles is independent of the exact vertex positions. 

As mentioned before, a polygonal knot with 6 vertices can only be an unknot or a trefoil knot, since these are the only knots with stick number less than or equal to 6~\cite{negamiRamseyTheoremsKnots1991,randellInvariantsPiecewiselinearKnots1998,millettKnottingRegularPolygons1994,knotinfo}. Calvo showed~\cite{calvoGeometricKnotTheory1998,calvoEmbeddingSpaceHexagonal2001} that these two types of knots can be differentiated using the algebraic intersection numbers of the triangles formed by 3 consecutive vertices.  In counting knot types, we identify a knot with its mirror image.

\begin{lem}[{Calvo~\cite{calvoGeometricKnotTheory1998,calvoEmbeddingSpaceHexagonal2001}}]\label{lem:intersections}
	For an oriented hexagonal knot $K$ in $\mathbb{R}^3$ whose vertices are in general position, let $\blacktriangle_j$ be the oriented triangle determined by the vertices $(x_{j-1}, x_j, x_{j+1})$ and let $\Delta_j$ be the signed transverse intersection count of the oriented polygon with the relative interior of $\blacktriangle_j$. Then: 
    \begin{itemize}
        \item $K$ is a right-handed trefoil if and only if $\Delta_j=1$ for all $j$.
        \item $K$ is a left-handed trefoil if and only if $\Delta_j=-1$ for all $j$.
        \item Otherwise $K$ is an unknot. Equivalently $\Delta_j=0$ for some j.
    \end{itemize}
\end{lem}

We also rely on the following lemma.

\begin{lem}[{cf. Huh–Jeon~\cite{huhKNOTSLINKSLINEAR2007}}]\label{lem:twointersection}
    Each edge of a hexagonal knot $K$ can transversely pierce the relative interiors of at most two triangles $\blacktriangle_j$; in particular, edge $e_i$ can only pierce $\blacktriangle_{i-2}$ and $\blacktriangle_{i+3}$, with indices taken modulo 6.  
\end{lem}

\begin{proof}
    Let $e_i$ be the edge between vertices $x_i$ and $x_{i+1}$. There are 6 edges $e_i$ and 6 triangles $\blacktriangle_j$. For ease of notation, assume all indices are taken mod 6. Edge $e_i$ cannot pierce the interior of $\blacktriangle_{i}$ or $\blacktriangle_{i+1}$ because it lies on the boundary of both triangles. It also cannot transversely pierce the interiors of either  of $\blacktriangle_{i-1}$ and $\blacktriangle_{i+2}$ since $e_i$ shares a vertex with each. So $e_i$ can pierce at most the two triangles $\blacktriangle_{i-2}$ and $\blacktriangle_{i+3}$.
\end{proof}

The following propositions use the previous lemmas to show exactly which knot types are possible when the convex hull of the points forms either type of polyhedron.

\begin{prop}
    Let $x_1,...,x_6 \in \mathbb{R}^3 $ be points in general position such that their convex hull forms a combinatorially regular octahedron. For a permutation $\sigma$ in $S_6$, let $K_\sigma$ be the polygonal knot with vertices $ x_{\sigma(1)}, x_{\sigma(2)}, \dots, x_{\sigma(6)} $. There are exactly 12 permutations (of the 720 in $S_6$) for which $K_\sigma$ is a trefoil knot, and all other knots are the unknot.
\label{prop:regoct}
\end{prop}

\begin{figure}[h!]
      \centering
      \includegraphics[width=0.25\linewidth]{figures/regularoctohedron.pdf}\qquad
      \includegraphics[width=0.25\linewidth]{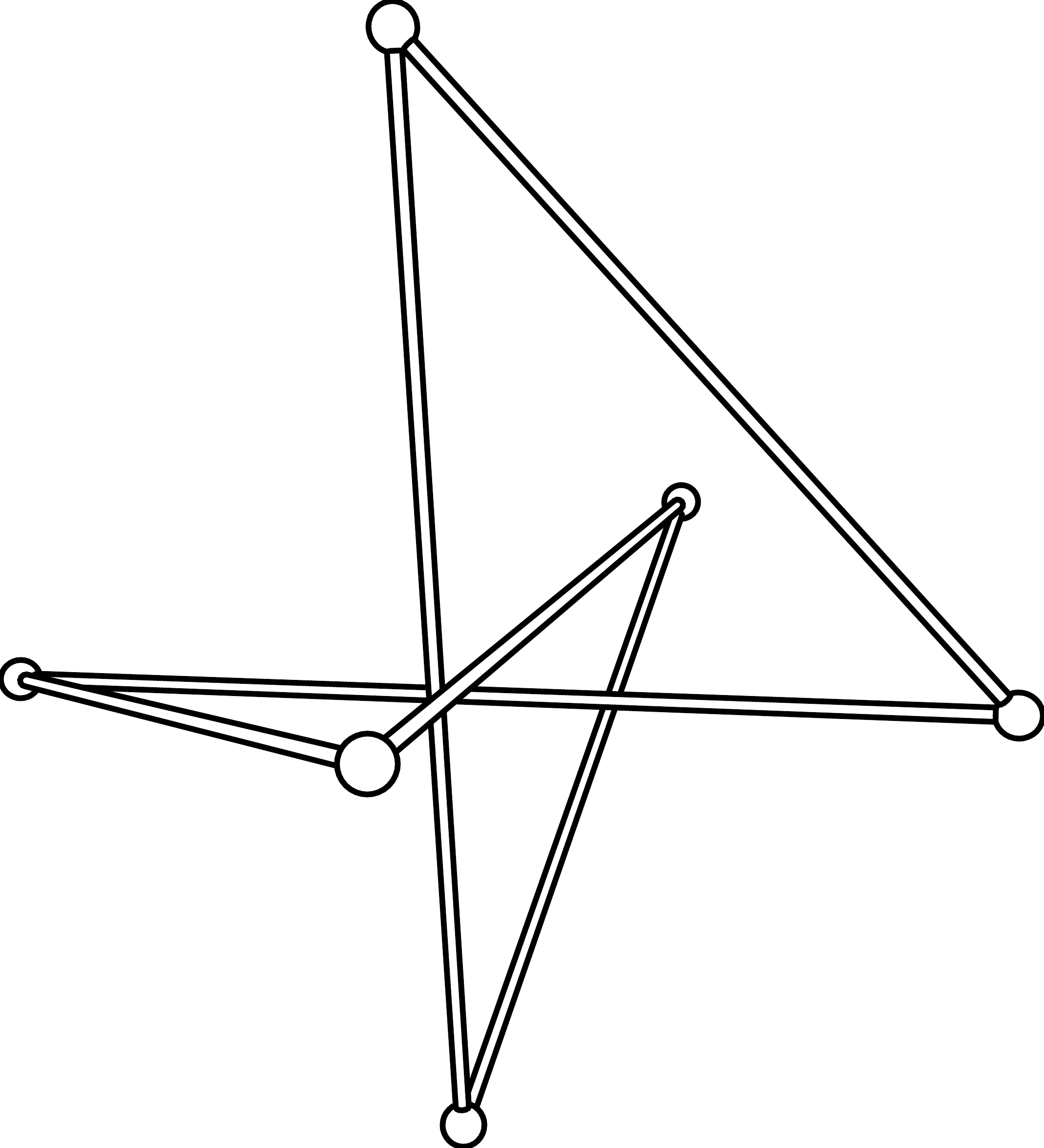}\qquad
      \includegraphics[width=0.25\linewidth]{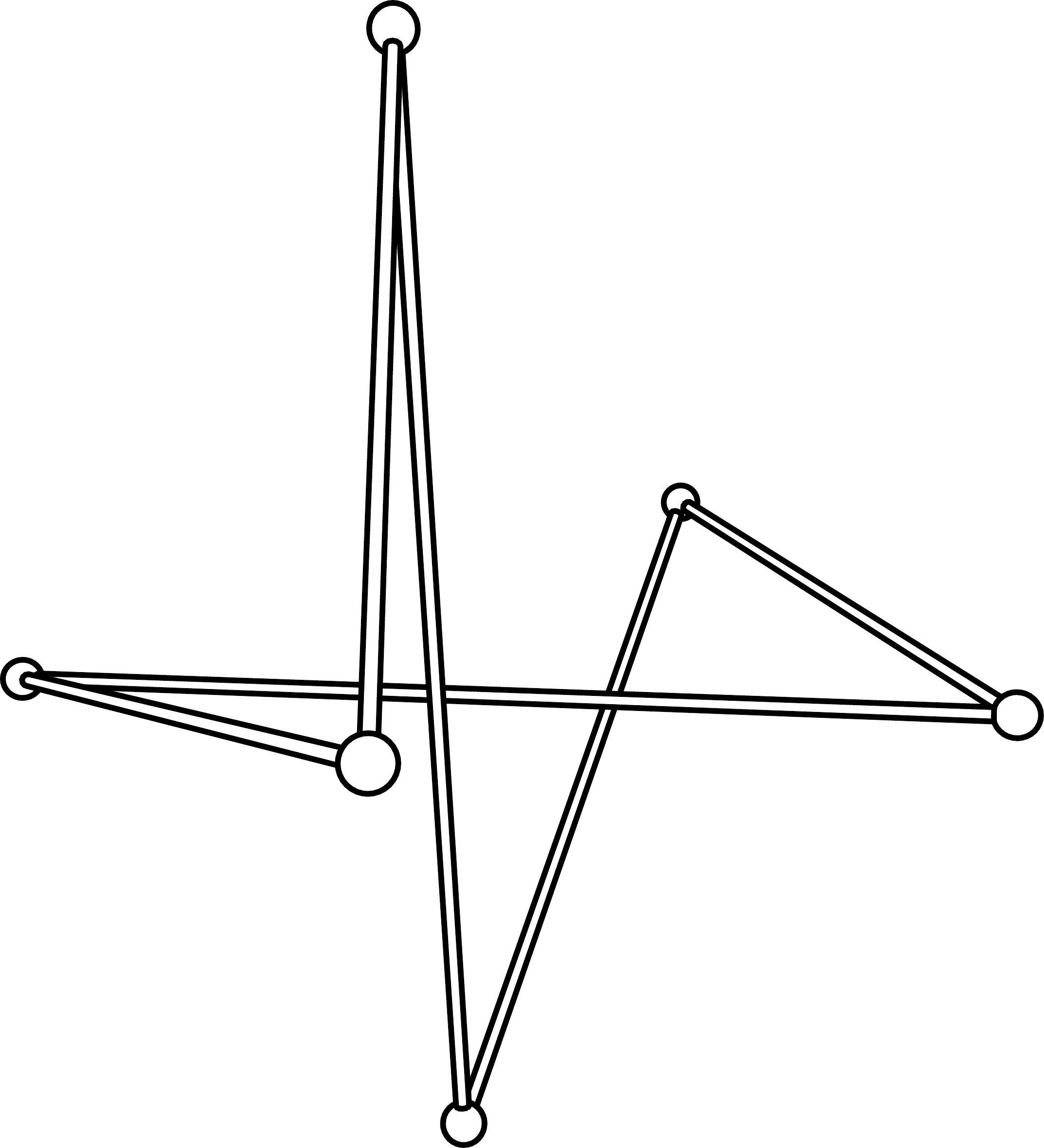}
      \caption{Both the unknot and trefoil knot are possible if the convex hull of the 6 points is a regular octahedron. The middle figure is a trefoil while the right is an unknot.}
      \label{fig:regocto}
  \end{figure}

\begin{proof}

    The 720 permutations yield 720 rooted, oriented polygonal representations $K_\sigma$. Notice that cyclically permuting the vertices of $K_\sigma$ or reversing their orientation does not affect the topological knot type; this only changes which vertex is listed first, and the orientation in which the knot is traversed.\footnote{Note however that reversing orientations can switch the ``curl" of a trefoil knot, defined in \cite{calvoGeometricKnotTheory1998}, changing its \emph{geometric} knot type.} Cyclically permuting vertices and reversing orientations are represented by elements of the dihedral group $D_6$. There are 60 cosets in $S_6/D_6$, and each coset consists of 12 elements. We identify orderings that differ only by cyclic relabeling or reversal; each resulting class represents one undirected six-cycle through the vertices. For each path that creates a trefoil knot, there will be 12 vertex orderings representing the same trefoil. 
	
	The fact that the trefoil path, if it exists, must be unique is due to Huh and Jeon~\cite{huhKNOTSLINKSLINEAR2007}. The fact that there exists an ordering of the vertices of a regular octahedron which gives a trefoil knot seems to be a folk theorem going back at least to Motzkin~\cite{motzkinCooperativeClassesFinite1967}, but we are not aware of an explicit proof in the literature, so we give one here.

  Consider a convex, regular octahedron, denoted $\mathcal{P}$, with vertices in general position, and a knot $K$, given by connecting vertices of this octahedron in some order. There are $\binom{6}{3} = 20$ triangles that can be formed from 3 vertices of $\mathcal{P}$. Of these triangles, 8 make up the boundary of $\mathcal{P}$. There are $\binom{6}{2} = 15$ segments connecting vertices of the octahedron, 12 of which are on the boundary. The boundary segments do not pierce any triangles. The remaining three segments are in the interior of the polyhedron, connecting each vertex to its ``opposite'' vertex (the vertex with which it does not share any edges of the octahedron). We refer to these edges as the ``diagonals'' of $\mathcal{P}$; see \cref{fig:regoctodiagonals}. In order for $K$ to be a  trefoil knot, \cref{lem:intersections} tells us 6 triangles must be pierced and \cref{lem:twointersection} tells us this will require at least 3 edges. Since the diagonals are the only possible edges that can pierce any triangles, any knot which does not include all three diagonals is necessarily the unknot.

  \begin{figure}[h!]
      \centering
      \includegraphics[width=0.3\linewidth]{figures/regularoctohedron.pdf}\qquad
      \includegraphics[width=0.3\linewidth]{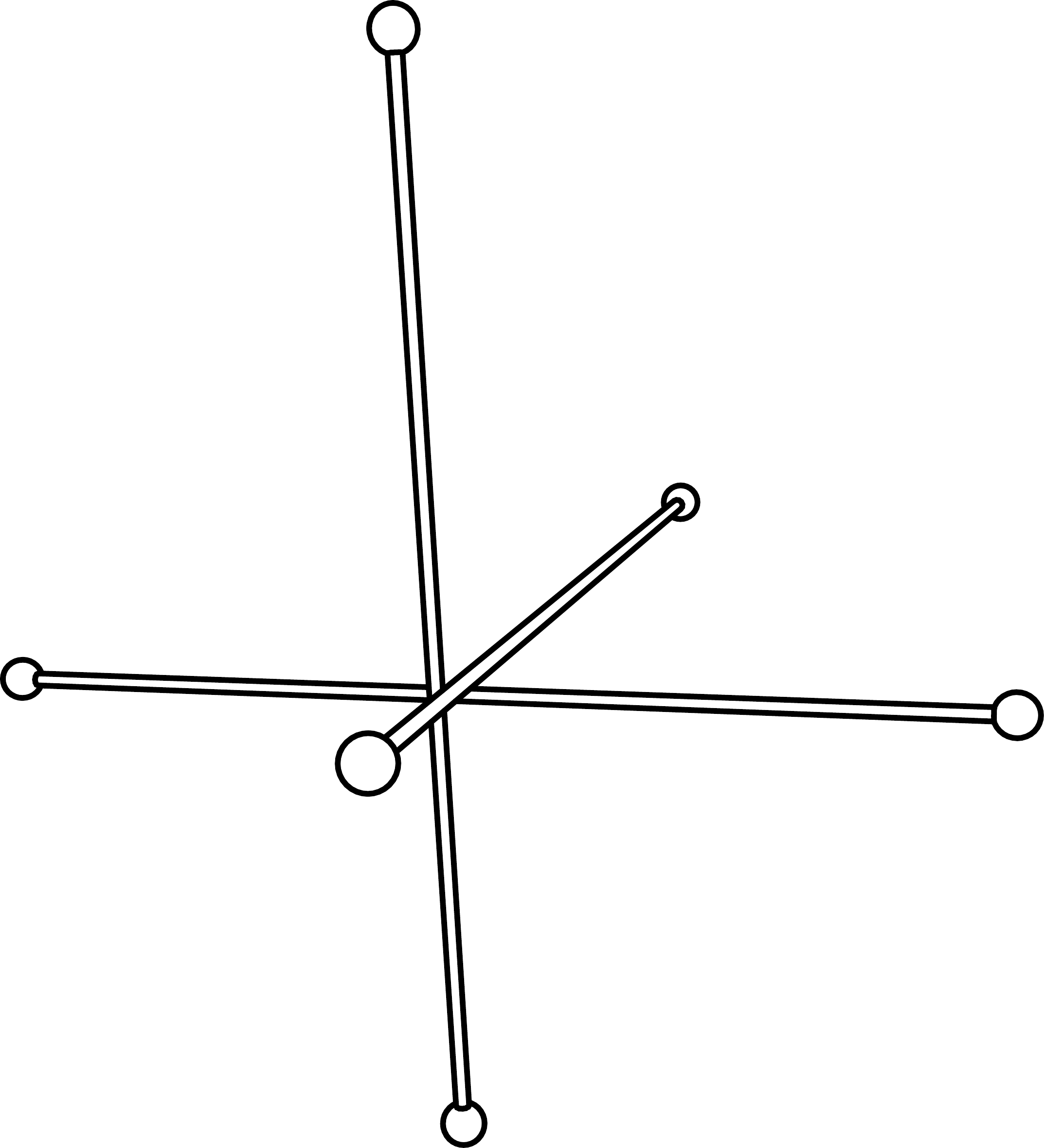}
      \caption{The diagonals of a convex regular octahedron. }
      \label{fig:regoctodiagonals}
  \end{figure}
  
    Because the diagonals go through the interior of the convex hull, each must cross the tetrahedron formed by the 4 vertices in $\mathcal{P}$ that do not define the diagonal. Since the vertices are in general position, the diagonal will pierce exactly 2 of the 4 faces of this tetrahedron. Each other diagonal either is an edge of the triangle or shares a vertex with it, and therefore cannot transversely pierce its interior. This is true for each diagonal: each will pierce exactly 2 of all the possible triangles. So there are exactly 6 triangles in the interior of the polyhedron which are pierced, two by each of the diagonals. See \cref{fig:diagintersect}.

  \begin{figure}[h!]
      \centering
      \includegraphics[width=0.3\linewidth]{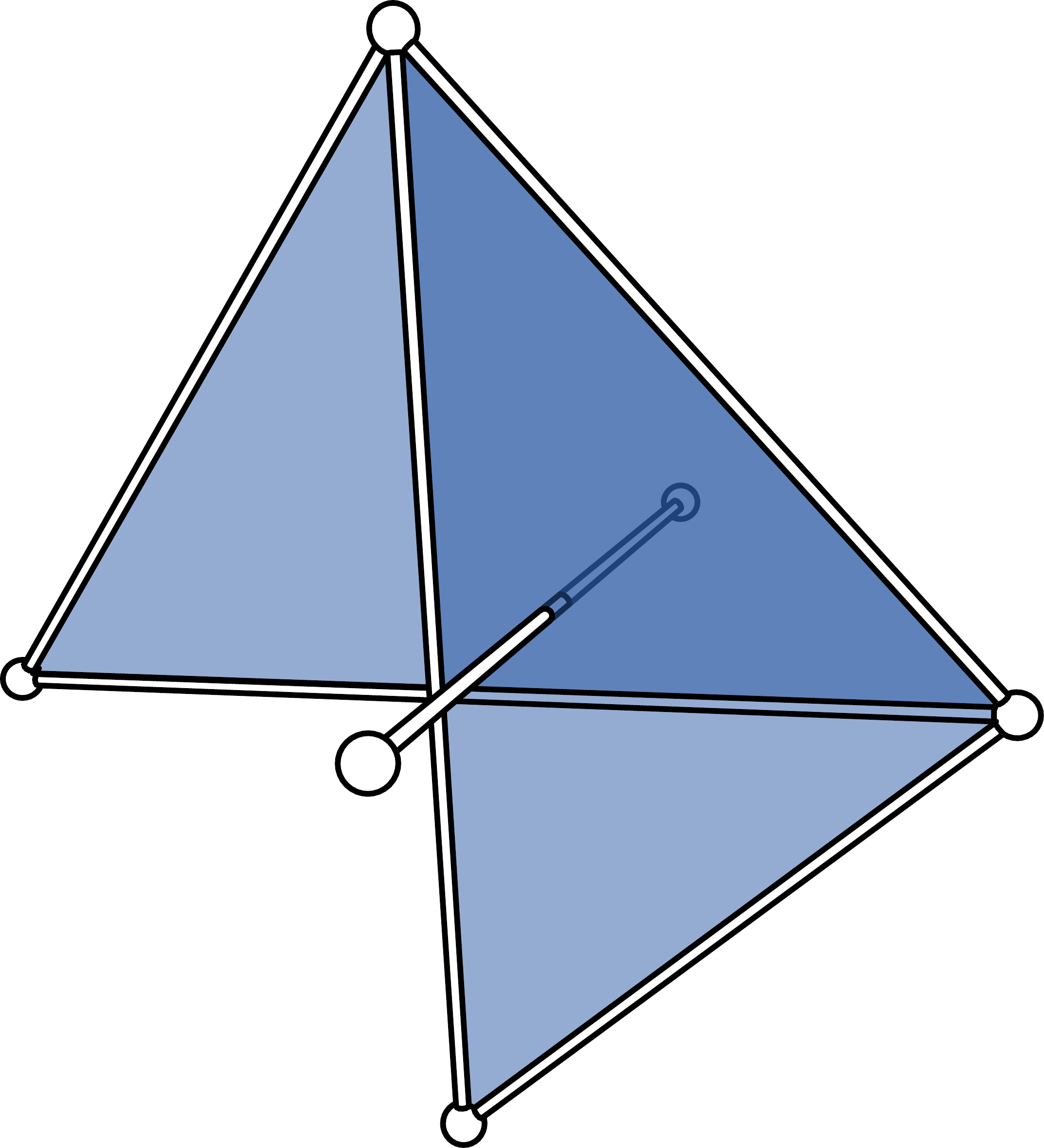}\qquad
      \includegraphics[width=0.3\linewidth]{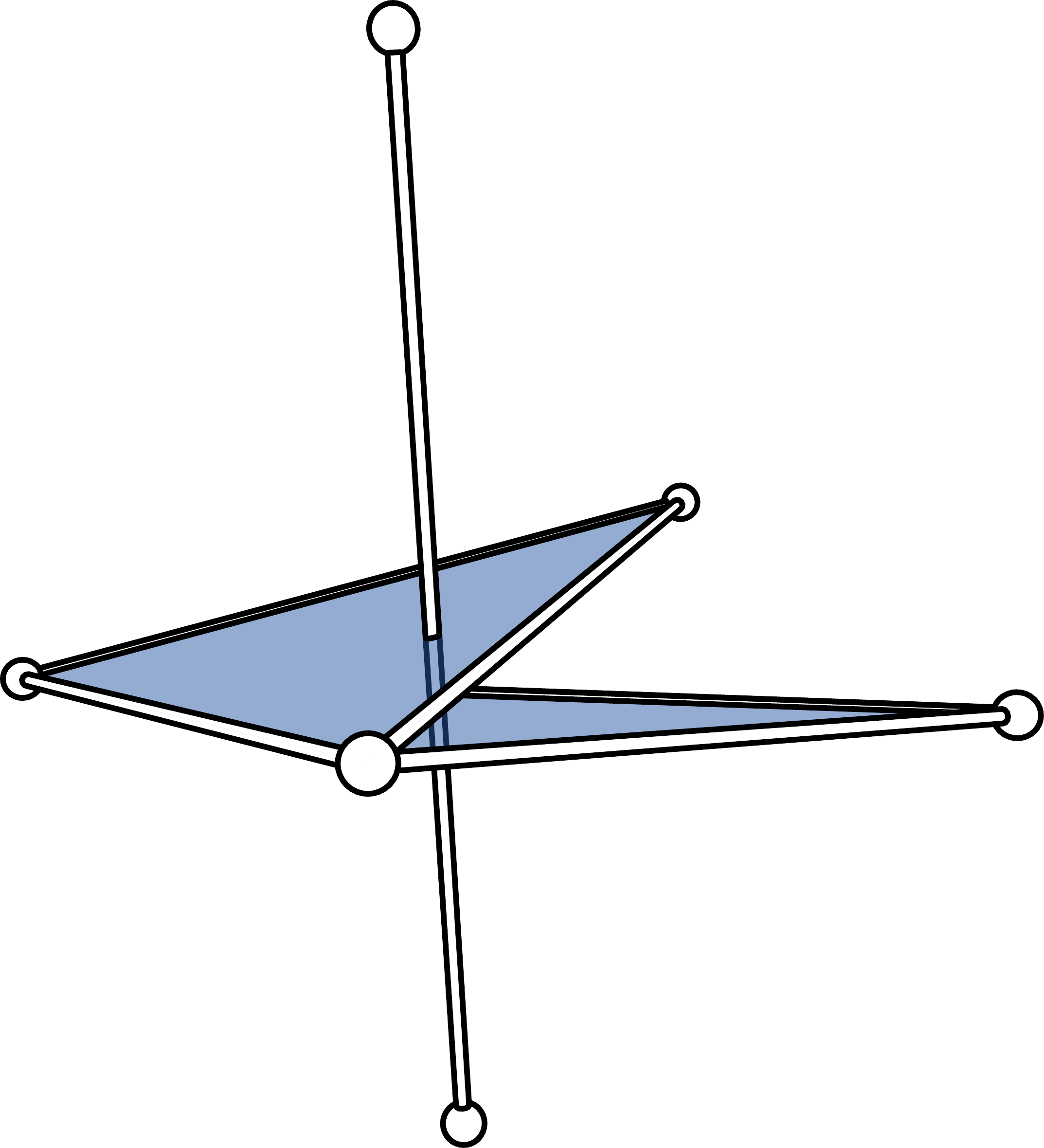}\qquad
      \includegraphics[width=0.3\linewidth]{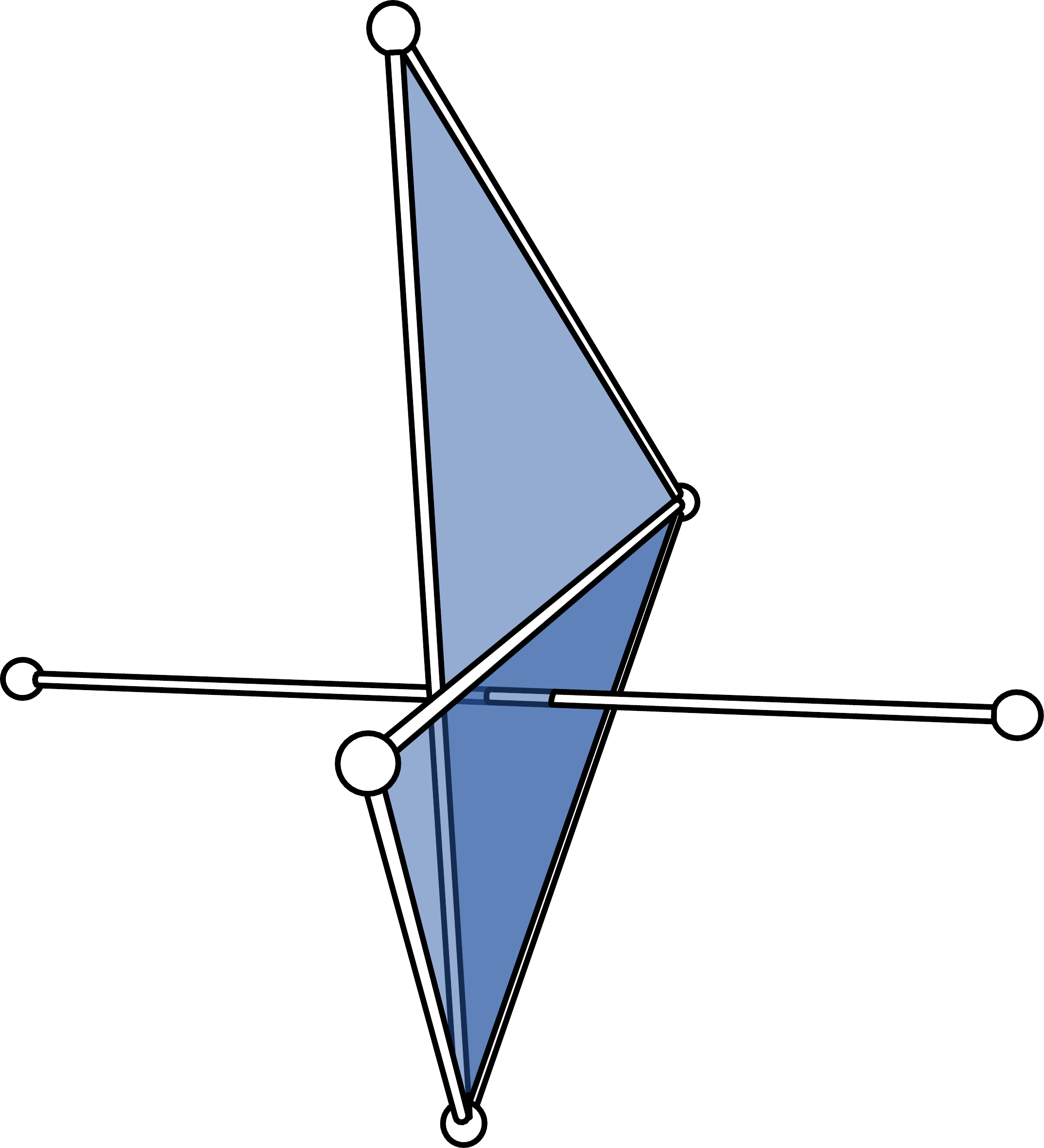}
      \caption{Each diagonal pierces the interior of two of the triangular faces of the tetrahedron which is the convex hull of the 4 other vertices. The faces which are pierced by the diagonal are filled in.}
      \label{fig:diagintersect}
  \end{figure}
  
  For a knot $K$ which has all three diagonals as edge segments, three other edges must be chosen to connect the diagonals into a closed loop. 
  To get a trefoil, the six triangles shown in \cref{fig:diagintersect} must be the six $\blacktriangle_j$, which can only happen if the three non-diagonal edges are the edges common to two of these triangles. 
  In other words, the only configuration which results in $K$ being a trefoil knot is the one in which we choose all of these shared edges as the segments to connect the diagonals. This results in the pierced triangles being exactly the six triangles $\blacktriangle_j$, thus making $\Delta_j=\pm 1$ for all~$j$.
  
  Since this configuration (shown in the middle of \cref{fig:regocto}) is unique, there is only one valid path which results in a trefoil knot. Thus there are exactly 12 permutations in $S_6$ for which $K_\sigma$ is a trefoil, the orbit of $D_6$ which is represented by this path. 
\end{proof}

By contrast, all knots formed by connecting vertices of the irregular octahedron, like the one shown in \cref{fig:irregknot}, must be unknots.

\begin{prop}
    Let $x_1,...,x_6 \in \mathbb{R}^3 $ be points in general position such that their convex hull forms an irregular octahedron. Let $K_\sigma$ be the polygonal knot with vertices $x_{\sigma(1)}, x_{\sigma(2)}, \dots, x_{\sigma(6)}$. For all permutations $\sigma \in S_6$, $K_\sigma$ is the unknot.
\label{prop:irregoct}
\end{prop}

\begin{figure}[h!]
    \centering
    \includegraphics[width=0.3\linewidth]{figures/irregularoctohedron.pdf}\qquad
    \includegraphics[width=0.3\linewidth]{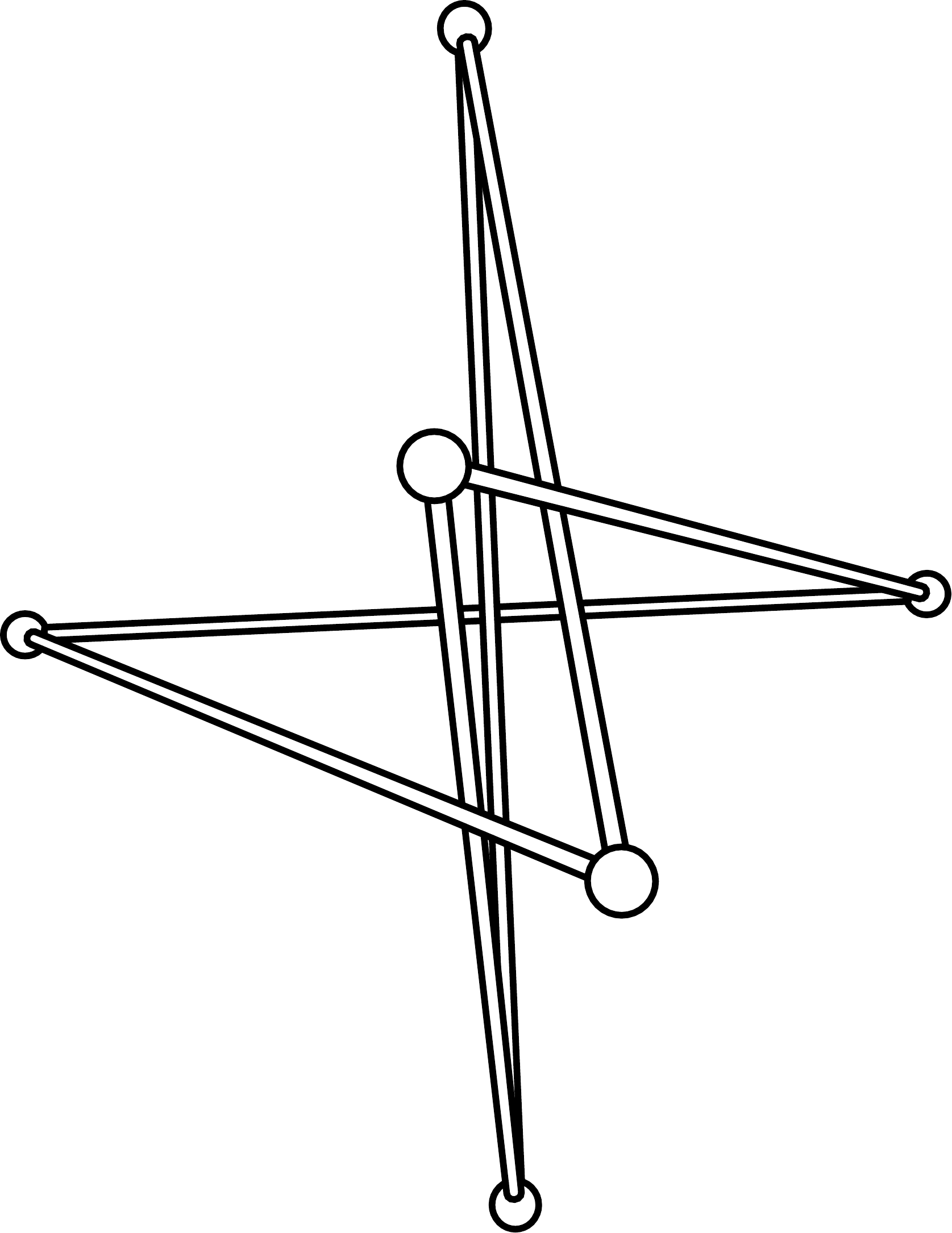}
    \caption{Every polygonal cycle through the vertices of the irregular octahedral type is the unknot.}
    \label{fig:irregknot}
\end{figure}

\begin{proof}
    As in the previous proof, 12 of the 15 possible edges are on the boundary of the convex hull. The three ``diagonals,'' shown in \cref{fig:irregknotdiagonals}, are in the interior of the polyhedron. Two of the diagonals connect a vertex of degree 4 to a vertex of degree 3, and one connects the two vertices of degree 3. As before, only the diagonals can pierce any of the possible triangles, so \cref{lem:intersections} and \cref{lem:twointersection} tell us any knot that does not use all three diagonals is an unknot. 

    Any knot that does contain all three diagonals will contain these three edges consecutively, as they share the degree 3 vertices. By \cref{lem:twointersection}, the two triangles formed by these three edges cannot be pierced by any diagonal, and there are no other possible edges which pierce any triangles. Therefore at least two triangles $\blacktriangle_j$ have $\Delta_j = 0$, so \cref{lem:intersections} tells us these knots are also unknotted.
\end{proof}

\begin{figure}[h!]
    \centering
    \includegraphics[width=0.3\linewidth]{figures/irregularoctohedron.pdf}\qquad
    \includegraphics[width=0.3\linewidth]{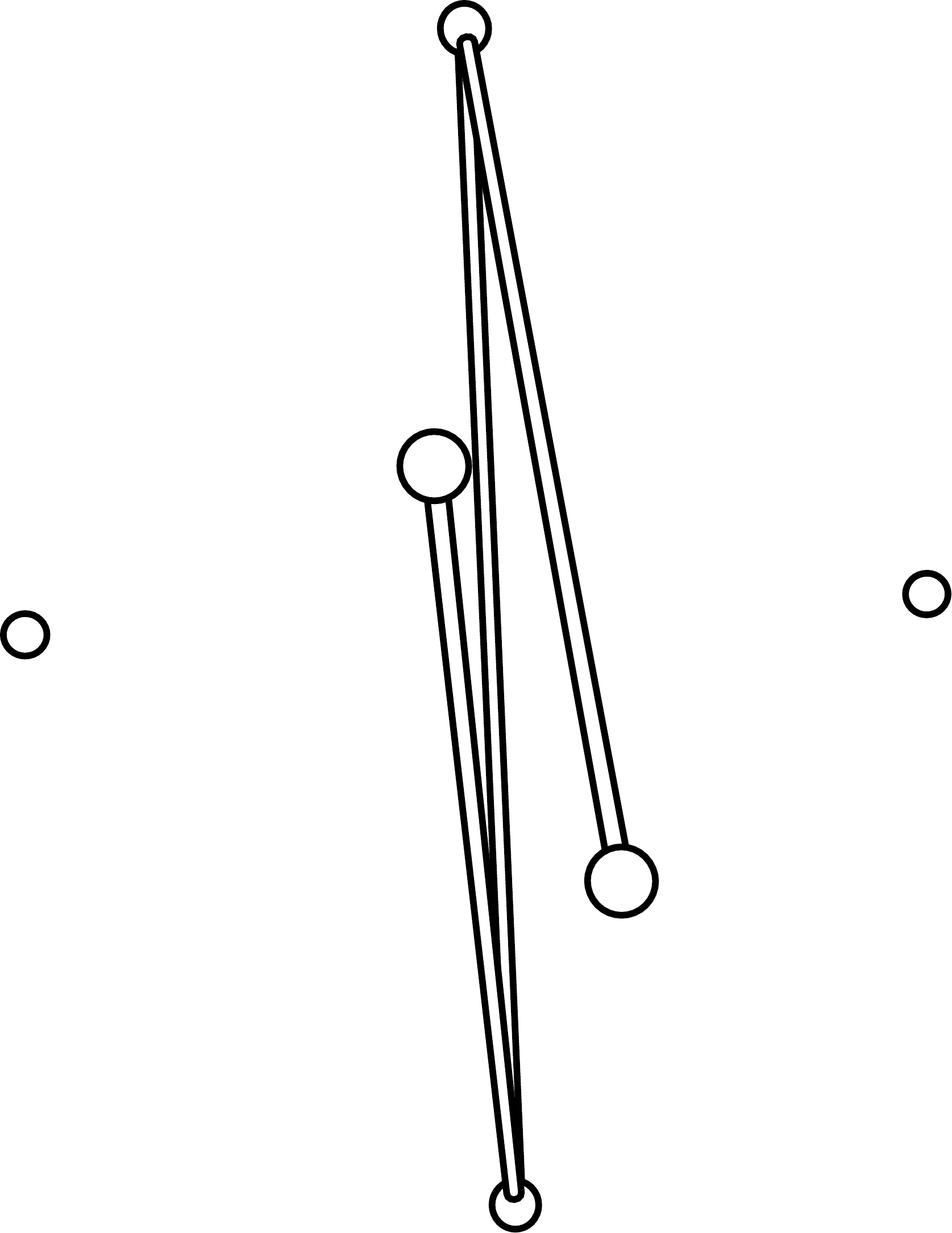}
    \caption{The three diagonals of the irregular octahedron form a path.}
    \label{fig:irregknotdiagonals}
\end{figure}

\begin{remark}
	The irregular octahedron is combinatorially equivalent to the cyclic polytope $C(6,3)$, which can be realized as the convex hull of six points  $x(t_1), \dots , x(t_6)$  on the moment curve $x(t) = (t,t^2,t^3) \in \R^3$. Huh and Jeon~\cite[\S 5]{huhKNOTSLINKSLINEAR2007} showed that $n$ points on the moment curve cannot be the vertices of a nontrivial knot with stick number $n$, which in particular implies that 6 points on the moment curve cannot form a trefoil. Hence, their result implies the conclusion of \cref{prop:irregoct} when the irregular octahedron is affinely equivalent to the standard embedding of $C(6,3)$.
\end{remark}

We can now combine \cref{prop:regoct,prop:irregoct} to see that \cref{thm:q} implies \cref{thm:probtrefoil}. Specifically, conditioning on whether the convex hull of the six points is a regular octahedron, \cref{prop:regoct} tells us that
    \[
        \mathbb{P}[\text{The polygon with vertices } \ x_1,x_2,\dots,x_6 \ \text{ is a trefoil} \ |\ \mathcal{H} \text{ is a regular octahedron}]= \frac{12}{720} = \frac{1}{60}.
    \]
    On the other hand, \cref{prop:irregoct} tells us that
    \[
        \mathbb{P}[\text{The polygon with vertices } \ x_1,x_2,\dots,x_6 \ \text{ is a trefoil} \ |\ \mathcal{H} \text{ is an irregular octahedron}] = 0,
    \]
    so by the law of total probability we have
    \[
        p \ = \ \frac{1}{60}\cdot  \mathbb{P}[\mathcal{H} \text{ is a regular octahedron}]+ 0 \cdot  \mathbb{P}[\mathcal{H} \text{ is an irregular octahedron}] = \frac{1}{60} \cdot q + 0 \cdot (1-q)  = \frac{q}{60}.
    \]

In order to calculate $p$, we therefore turn our attention to calculating $q$.

\section{Calculating \texorpdfstring{$q$}{q} in the plane}
\label{sec:calculating q}

The strategy for computing $q$, which we owe to Felix Pahl~\cite{jorikiAnswer2024}, is to convert to a planar problem by stereographic projection. Specifically, if $x_1, \dots , x_6$ are six uniform points on the sphere with convex hull $\mathcal{H}$, we will stereographically project to the plane $x_6^\perp$. By rotational invariance of the uniform measure on the sphere, we can assume $x_6$ is the north pole, so that $x_6^\perp$ is the $xy$-plane. 

The resulting 5 points in the plane will have a convex hull $\mathcal{C}$, and the key to this strategy is that $\mathcal{C}$ will be a $k$-gon if and only if $x_6$ is a vertex of degree $k$ in $\mathcal{H}$, which in turn gives information about the combinatorial type of $\mathcal{H}$. For $k=3,4,5$, let $p_k$ denote the probability that $\mathcal{C}$ is a $k$-gon. By conditioning on the two types of octahedron, we can say:
    \begin{align*}
        p_3 &= \frac{2}{6} \cdot (1-q) + 0 \cdot q = \frac{1-q}{3} \\
        p_4 &= \frac{2}{6} \cdot (1-q) + 1 \cdot q = \frac{1+2q}{3} \\
        p_5 &= \frac{2}{6} \cdot (1-q) + 0 \cdot q = \frac{1-q}{3}. 
    \end{align*}

We record this observation in the following lemma.

\begin{lem}[Pahl~\cite{jorikiAnswer2024}]\label{lem:q and p3}
     Let $\mu$ be the pushforward measure of the uniform measure on the sphere to the plane via stereographic projection. Let $p_k$ be the probability that the convex hull $\mathcal{C}$ of 5 random points drawn independently from $\mu$ is a $k$-gon. Then 
	 \[
	 	p_3 = p_5 = \frac{1-q}{3} \quad \text{and} \quad p_4 = \frac{1+2q}{3}.
	 \]
\end{lem}

	Rearranging, we in particular know that $q = 1-3p_3$, so we can compute $q$ if we know $p_3$, that is, the probability that $\mathcal{C}$ is a triangle. Now, the event that $\mathcal{C}$ is a triangle can be split into $\binom{5}{3}=10$ disjoint events, in which any three of the 5 points are the extremal points of the convex hull. Let $r$ be the probability that the triangle determined by the stereographic projection of 3 uniformly chosen points on the sphere contains the image of two additional uniformly chosen points. Then we know that $p_3 = 10 r$, so $q = 1-30 r$.

    Conditional on the randomly chosen triangle $\mathcal{T}$, the probability that one additional point lies in $\mathcal{T}$ is precisely the probability content of $T$, that is, $X:=\mu(\mathcal{T})$. Since the two additional points are conditionally independent, the probability that both lie in $T$ is $X^2$. Averaging over triangles, we obtain $r=\mathbb E[X^2]$ and therefore $q = 1-30\mathbb{E}[X^2]$. Again, we record this observation in a lemma.

\begin{lem}[Pahl~\cite{jorikiAnswer2024}]\label{lem:p3 and probability content}
     Let $\mu$ be the pushforward measure of the uniform measure on the sphere to the plane via stereographic projection. If three points are drawn independently from $\mu$, forming a triangle $\mathcal{T}$, let $X=\mu(\mathcal{T})$. Then $q = 1-30 \mathbb{E}[X^2] $.
\end{lem}

\section{Calculation of \texorpdfstring{$\mathbb{E}[X^2]$}{E[X²]}}
\label{sec:calculation}

We have now reduced everything to the calculation of $\mathbb{E}[X^2]$. First, we need to determine the density of the measure~$\mu$. Consider the unit sphere $\mathbb S^2 \subset \mathbb R^3$. Stereographic projection from the north pole onto the $xy$-plane pushes the uniform probability measure on $\mathbb S^2$ forward to the measure $$d\mu(x,y) = \frac{1}{\pi(x^2+y^2+1)^2}\,dx\, dy.$$

Let $\mathcal{T}=\operatorname{conv}\{a,b,c\}$ be the triangle determined by three points $a,b,c$ in $\mathbb R^2$. The probability content of $\mathcal{T}$ with respect to $\mu$ is $$P(a,b,c)=\mu(\mathcal{T})=\int_\mathcal{T} d\mu(x,y)=\int_\mathcal{T} \frac{1}{\pi(x^2+y^2+1)^2}\,dx\, dy.$$

Consider the 1-form $\omega=\frac{x\, dy - y\, dx}{2\pi (1+x^2 +y^2)}$. A direct computation shows that $d\omega=\frac{1}{\pi(x^2+y^2+1)^2}\,dx \wedge dy$, which is the density form associated with $\mu$. This will allow us to use Stokes' theorem to convert the integral over $\mathcal{T}$ to a line integral over the boundary.

Given an ordered pair  $(a,b)\in \mathbb \R^2 \times \R^2$, let $H(a,b)$ denote the oriented line segment from $a$ to $b$ and define $$L(a,b):=\int_{H(a,b)} \omega.$$ Reversing the orientation of the segment gives $L(a,b)=-L(b,a)$. By Stokes' theorem, we conclude that $$ P(a,b,c)=|L(a,b)+L(b,c)+L(c,a)|.$$

If $a,b,c$ are sampled independently according to $\mu$, we would like to find a formula for the expected value of $P(a,b,c)^2$, i.e., $\mathbb E[P(a,b,c)^2]$. Note that the measure $\mu$ is radially symmetric and recall that $L(a,b)=-L(b,a)$. For a fixed $b$, reflection across the line through $(0,0)$ and $b$ preserves $\mu$, fixes $b$, and changes the sign of $\omega$. Therefore, holding $b$ fixed and letting $a$ and $c$ vary, we get $$\mathbb E_a[L(a,b)]=\mathbb E_c[L(b,c)]=0.$$ 

Since $a$ and $c$ are conditionally independent given $b$, we get $\mathbb E[L(a,b)L(b,c)\ | \ b]=0$ and hence ${\mathbb E[L(a,b)L(b,c)]=0}$. By cyclic symmetry, the other two mixed expectations vanish as well. From this observation and linearity of expectation, it follows that $$\mathbb E[P(a,b,c)^2]=\mathbb E[(L(a,b)+L(b,c)+L(c,a))^2]=\mathbb E[L(a,b)^2]+\mathbb E[L(b,c)^2]+\mathbb E[L(c,a)^2]=3\mathbb E[L(a,b)^2].$$ Thus, our new goal is to compute $\mathbb E[L(a,b)^2]$ with $a,b$ sampled independently according to $\mu$.

Let $a=(x_1,y_1)$ and $b=(x_2,y_2)$. We can parameterize the line segment between $a$ and $b$ by $(x_1,y_1)+t(x_2-x_1,y_2-y_1)$ with $t$ going from 0 to 1. If we plug  the expressions $x=x_1+t(x_2-x_1)$ and $y=y_1+t(y_2-y_1)$ into the 1-form $\omega$, then we get the following expression: 
\begin{align*}
	L(a,b) & = \int_0^1 \frac{(x_1 + t(x_2 - x_1))(y_2 - y_1) - (y_1 + t(y_2 - y_1))(x_2 - x_1)}{2\pi (1 + (x_1 + t(x_2 - x_1))^2 + (y_1 + t(y_2 - y_1))^2)} \ dt \\
	& = \int_0^1 \frac{x_1 y_2-x_2 y_1}{2 \pi \left(1 + x_1^2 + y_1^2 + 2t(x_1(x_2-x_1) + y_1(y_2-y_1)) + t^2 ((x_2-x_1)^2 + (y_2-y_1)^2)\right)} \ dt.
\end{align*} 

By completing the square and rearranging, the integrand is clearly the derivative of an expression involving $\arctan$ (cf.~\cite[eq. 2.103.4]{gradshteynTableIntegralsSeries2015}). Using the Fundamental Theorem of Calculus, we conclude that
\[
	L(a,b) = L(x_1,y_1,x_2,y_2) = \frac{x_1y_2-x_2y_1}{2\pi B} \left(\arctan \left(\frac{x_1^2+y_1^2-x_1 x_2 - y_1 y_2}{B}\right) - \arctan \left(\frac{-x_2^2-y_2^2+x_1 x_2 + y_1 y_2}{B}\right)\right),
\]
where $B= \sqrt{(x_1y_2-x_2y_1)^2+(x_1-x_2)^2+(y_1-y_2)^2}$.

    In order to compute $\mathbb E[L(a,b)^2]$ with $a,b$ sampled independently according to $\mu$, we need to integrate $L(a,b)^2$ weighted by the probability measure given by $\mu$. This leads us to the formula $$\mathbb E[L(a,b)^2]=\int_{\mathbb R^2} \int_{\mathbb R^2} \frac{L(x_1,y_1,x_2,y_2)^2}{\pi\, (1+x_1^2+y_1^2)^2\  \pi\, (1+x_2^2+y_2^2)^2} \, dx_1\, dy_1\, dx_2\, dy_2.$$

    The Blaschke–Petkantschin substitution decomposes a pair of points $a,b \in \mathbb R^2$ into the affine line $\ell$ containing them and their one-dimensional coordinates along $\ell$. Let $z=(\cos(\phi), \sin(\phi))$. Let $w=(-\sin(\phi), \cos(\phi))$. Let $h,s,t, \phi$ be parameters. The Blaschke–Petkantschin formula is $$\int_{\mathbb R^2}\int_{\mathbb R^2} f(x_1,y_1,x_2,y_2)\, dx_1\, dy_1\, dx_2\, dy_2 = \int_0^{\pi} \int_{{\mathbb R}^3} f(hz+sw,hz+tw)\, |t-s|\, ds\, dt\, dh \, d\phi.$$ 

     If we let $F$ be defined by $$F(x_1,y_1,x_2,y_2)=\frac{L(x_1,y_1,x_2,y_2)^2}{\pi\, (1+x_1^2+y_1^2)^2\  \pi\, (1+x_2^2+y_2^2)^2},$$ then our goal is to compute $$\mathbb E[L(a,b)^2]=\int_{\mathbb R^2} \int_{\mathbb R^2} F(x_1,y_1,x_2,y_2)\, dx_1\, dy_1\, dx_2\, dy_2.$$ The Blaschke–Petkantschin coordinate parameterization is $$x_1=h\cos(\phi)-s\sin(\phi),\quad y_1=h\sin(\phi)+s\cos(\phi),\quad x_2=h\cos(\phi)-t\sin(\phi),\quad y_2=h\sin(\phi)+t\cos(\phi).$$
     When this substitution is made and we simplify the resulting expression, the integrand is independent of $\phi$! The expression simplifies to the following triple integral:
     $$\mathbb E[L(a,b)^2]=\int_{{\mathbb R}^3} \cfrac{h^2 \, |t-s|\, \left( \arctan\left(\frac{t}{\sqrt{1+h^2}}\right)- \ \arctan\left(\frac{s}{\sqrt{1+h^2}}\right)\right)^2}{4\, \pi^3 \, (1+h^2+t^2)^2\, (1+h^2+s^2)^2\, (1+h^2)} \, ds \, dt \, dh.$$

	 
	 Making the substitutions $A=\sqrt{1+h^2}$, $s=Au$, and $t=Av$  reduces the integral to $$\mathbb E[L(a,b)^2]=\frac{1}{15\pi^3} \int_{\mathbb R^2} \frac{|u-v|\, (\arctan u - \arctan v)^2}{(1+u^2)^2\, (1+v^2)^2} \, du \, dv.$$ Next make the substitution $u=\tan \alpha, v=\tan \beta$ to reduce to 
     $$\mathbb E[L(a,b)^2]=\frac{1}{15\pi^3} \int_{-\frac{\pi}{2}}^{\frac{\pi}{2}}\int_{-\frac{\pi}{2}}^{\frac{\pi}{2}} \cos(\alpha)\, \cos(\beta)\, | \sin(\alpha-\beta)|\, (\alpha-\beta)^2 \, d\alpha \ d\beta.$$ This is an elementary integral whose value can be computed to be
	 \[\mathbb E[L(a,b)^2]=\frac{1}{15\ \pi^3} \left(\frac{\pi^3}{6} - \frac{5\ \pi}{8}\right) = \frac{1}{90} - \frac{1}{24\pi^2}.\]
	 
	 \subsection{Proving the theorems}
	 Since $\mathbb E[P(a,b,c)^2]=3\ \mathbb E[L(a,b)^2]$, we conclude that $$\mathbb E[X^2] = \mathbb E[P(a,b,c)^2]=\frac{1}{30} - \frac{1}{8\pi^2}.$$
	 Therefore, 
	 \[
	 	q= 1-30 \mathbb{E}[X^2] = \frac{15}{4\pi^2},
	 \]
	 proving \cref{thm:q}. \cref{thm:beta vertex count probs} follows immediately using \cref{lem:q and p3}. Finally, as we saw at the end of \cref{sec:knots and polyhedra},
	 \[
	 	p = \frac{q}{60} = \frac{1}{16\pi^2},
	 \]
	 proving \cref{thm:probtrefoil}.
	 
\section{Discussion} 
\label{sec:discussion}

Having computed knotting probabilities for polygonal knots determined by six random points on the sphere, an obvious next step is to determine knotting probabilities for knots determined by $n$ points on the sphere for larger $n$. However, even for $n=7$ the number of knots of different types in the permutation orbit is not completely determined by the combinatorics of the convex hull, so it is not enough to compute the probabilities of the different combinatorial types of the convex hull.\footnote{For those interested in this problem, see Calvo~\cite{calvoGeometricKnotTheory1998,calvoGeometricKnotSpaces2001} and Huh~\cite{huhHeptagonalKnotsRadon2011} for geometric characterizations of heptagonal knots in the spirit of \cref{lem:intersections}.}

That said, determining the probabilities of different combinatorial types of the convex hull is also an interesting question for any $n > 6$, though the number of combinatorial types grows quite rapidly with $n$; see \seqnum{A000109} in OEIS~\cite{oeis}. Even for $n=7$ there are five distinct combinatorial types that occur with positive probability, so the reduction to one hull type probability that we used in this paper no longer applies directly.

In our planar reduction we computed the probability mass function of the number of vertices of the convex hull of five points independently sampled from the density $\frac{1}{\pi(1+x^2+y^2)^2}$. An interesting feature of this measure is that $p_3=p_5$; are there any other natural measures for which this is true? One natural generalization of our measure is the family of beta-prime distributions, which have density $\frac{\beta-1}{\pi  \ (1+x^2+y^2)^\beta}$ for $\beta > 1$. Since techniques which work for beta-prime polytopes are also often useful for beta polytopes, another class of measures to consider is the family of beta distributions, which are supported on the unit disk and have density $\frac{\beta+1}{\pi}(1-x^2-y^2)^\beta$ for $\beta > -1$. 

Returning to the original question about knot probabilities, it would be interesting to determine the exact probability of knotting, even for hexagons, in \emph{any} other random knot model with a continuous state space.

\

\section*{Acknowledgments} 
\label{sec:acknowledgments}
\noindent Thanks to Felix Pahl for inspiration and to Jason Cantarella for helpful conversations.

\section*{Disclaimer}

\noindent Sandia National Laboratories is a multimission laboratory managed and operated by National Technology \& Engineering Solutions of Sandia, LLC, a wholly owned subsidiary of Honeywell International Inc., for the U.S. Department of Energy’s National Nuclear Security Administration under contract DE-NA0003525.
This paper describes objective technical results and analysis. Any subjective views or opinions that might be expressed in the paper do not necessarily represent the views of the U.S. Department of Energy or the United States Government.

\section*{Declaration of AI Usage}

The authors used various AI models as search tools. AI was not used for writing or editing, nor was it used to generate any of the results in this paper.

\printbibliography

\end{document}